\documentclass[12pt,twoside,reqno]{amsart}

\usepackage[OT1]{fontenc}
\usepackage{type1cm}

\usepackage{bbm}
\usepackage{enumerate}
\usepackage{graphicx}
\usepackage{float}

\usepackage{amsthm}
\RequirePackage{amsmath,amsfonts,amssymb,amsthm}

\usepackage{xcolor}
\usepackage[pagebackref]{hyperref}
\hypersetup{
   colorlinks,
    linkcolor={red!60!black},
    citecolor={blue!60!black},
    urlcolor={blue!90!black}
}

\usepackage{psfrag}
  \numberwithin{equation}{section}

\usepackage{booktabs}

  \newcommand{\N}{\mathbb{N}}         
  \newcommand{\R}{\mathbb{R}}         
  \newcommand{\Z}{\mathbb{Z}}
  \newcommand{\EE}{\mathbb{E}}        
  \newcommand{\PP}{\mathbb{P}}        

  \newtheorem{thm}{Theorem}[section]
  \newtheorem{lemma}[thm]{Lemma}

  \theoremstyle{remark}
  \newtheorem{rem}[thm]{Remark}

  \DeclareMathOperator{\spt}{spt}

\begin{document}

\title{New bounds on Cantor maximal operators}

\author{Pablo Shmerkin}
\email{pshmerkin@math.ubc.ca}
\address{Department of Mathematics, the University of British Columbia, 1984 Mathematics Road
Vancouver, BC, V6T 1Z2, Canada}

\author{Ville Suomala}
\email{ville.suomala@oulu.fi}
\address{Research Unit of Mathematical Sciences,
	P.O.Box 8000, FI-90014,  University of Oulu,
	Finland}

\thanks{PS was supported by project PICT 2015-3675 (ANPCyT) and by an NSERC discovery grant. VS was in part supported by the Academy of Finland.  We also acknowledge support from the Institut Mittag-Leffler via the ``Fractal Geometry and Dynamics'' research program, where this project started.}

\subjclass[2020]{Primary: 42B25; Secondary: 28A80, 60G57.}

\begin{abstract}
We prove $L^p$ bounds for the maximal operators associated to an Ahlfors-regular variant of fractal percolation. Our bounds improve upon those obtained by I. {\L}aba and M. Pramanik and in some cases are sharp up to the endpoint. A consequence of our main result is that there exist Ahlfors-regular Salem Cantor sets of any dimension $>1/2$ such that the associated maximal operator is bounded on $L^2(\R)$. We follow the overall scheme of {\L}aba-Pramanik for the analytic part of the argument, while the probabilistic part is instead inspired by our earlier work on intersection properties of random measures. \end{abstract}

\maketitle

\section{Introduction}

\subsection{Maximal functions associated to singular measures}

One of the most classical results in real analysis is the $L^p$ boundedness of the Hardy-Littlewood maximal operator, which can be restated as follows: let $\nu$ be Lebesgue measure on the unit ball of $\R^d$. Then the maximal operator
\[
\overline{M}_\nu f(x) = \sup_{r>0} \int |f(x+ry)|\,d\nu(y)
\]
is bounded on $L^p$ for $p>1$. It makes sense to study such operators $\overline{M}_\nu$ also when $\nu$ is replaced by other, singular measures; one would expect its boundedness properties to reflect in some sense the geometry of the measure $\nu$. The case in which $\nu$ is surface area on the $(d-1)$-dimensional sphere is the celebrated spherical maximal theorem of Stein \cite{Stein1976}, for the case $d\ge 3$, and Bourgain \cite{Bourgain1986}, in the more challenging case $d=2$: in this case $\overline{M}_\nu$ is bounded on $L^p$ if and only if $p>d/(d-1)$. A large body of related work exists in which $\nu$ is replaced by Hausdorff measure on more general manifolds under curvature assumptions, or is assumed to satisfy a power Fourier decay bound; see for example \cite{RubioDeFrancia1986}.

Neither of the classical approaches gives information if $\nu$ is a singular measure on the real line, since the concept of submanifold or curvature is not available and the required Fourier decay cannot possibly hold. Nevertheless, it is natural to study this problem when $\nu$ is, for example, Hausdorff measure on a Cantor set of dimension $<1$. A first breakthrough in this direction was achieved by {\L}aba and Pramanik in \cite{LabaPramanik2011}. In order to state their result, we introduce the \emph{restricted} (or single-scale) version of the maximal operator, defined as
\[
M_\nu f(x) = \sup_{r\in [1,2]}\int |f(x+ry)|\,d\nu(y).
\]
The restricted version is easier to handle technically and, in any case, its mapping properties can be used in some cases to derive bounds also for certain unrestricted operators - see \S\ref{subsec:mainresult} below.

\begin{thm}[{\cite[Theorem 1.3]{LabaPramanik2011}}] \label{thm:LabaPramanik}
For any $s\in (2/3,1)$ there exists a measure $\nu$ supported on a Cantor set of Hausdorff dimension $s$, such that the associated restricted maximal operator $M_\nu f(x)$ is bounded from $L^p(\R)$ to $L^q(\R)$ for $p>(2-s)/s$ and $q\in[p,ps/(2-2s)]$.
\end{thm}

 The Cantor set and the measure $\nu$ arising in the proof of Theorem \ref{thm:LabaPramanik} are obtained through an ad-hoc random iterative process; no almost sure statements with respect to an underlying distribution on Cantor measures are made.

In subsequent work, {\L}aba \cite{Laba2018} considered maximal operators for certain self-similar Cantor sets in which the randomization only occurs in the first level pattern of the construction. While stopping short of proving $L^p$ estimates in this case, her results cover Cantor sets of arbitrarily small dimension, in addition to providing the first results in the area for self-similar examples.

In order to discuss the sharpness of Theorem \ref{thm:LabaPramanik}, let us note the following simple lemma.

\begin{lemma}\label{lem:sharpness}
	If $\nu$ is a finite measure giving positive mass to a set of Hausdorff dimension $<t$, then $M_\nu$ cannot be bounded from $L^{1/t}(\R)$ to $L^q(\R)$ for any $q\in [1,+\infty]$.
\end{lemma}

\begin{proof}
 If $\nu(A)>0$ for a set $A$ of Hausdorff dimension $<t$, then by the mass distribution principle (see e.g. \cite[Proposition 4.9]{Falconer2014}) there are $s<t$ and a point $x_0\neq 0$ in the support of $\nu$ such that for a sequence $\delta_i\downarrow 0$, we have \[
 \nu([x_0-\delta_i,x_0+\delta_i])>\delta_i^s.
 \]
 Let $f_i$ be the indicator of the interval $[x_0-2\delta_i,x_0+2\delta_i]$. If $x\in [-x_0,0]$ then taking $r=1-x/x_0$ in the definition of $M_\nu$ we see that $M_\nu f_i(x) > \delta_i^s$, and hence $\|M_\nu f_i\|_{L^q(\R)}> x_0^{1/q} \delta_{i}^{s}$. On the other hand, $\|f_i\|_{L^{p}(\R)}\le (4\delta_i)^{1/p}$. Taking $\delta_i\downarrow 0$, we see that $M_\nu$ cannot be bounded from $L^p(\R)$ to $L^q(\R)$ if $p\le 1/t$.
\end{proof}

Note that for all dimension values $s\in (2/3,1)$, there is a gap between the admissible values $p>(2-s)/s$ provided by Theorem \ref{thm:LabaPramanik} and the barrier $p\ge 1/s$ arising from Lemma \ref{lem:sharpness}. It is natural to ask what is the optimal range of $p$, given $s$, see \cite[Remark 3 on p. 350]{LabaPramanik2011}.

In this article we obtain a version of the theorem of {\L}aba and Pramanik with improved exponents and dimension bounds, which in some cases close the gap indicated above. In particular, we show that there are measures supported on Ahlfors-regular Salem Cantor sets of any dimension $>1/2$ satisfying non-trivial maximal operator bounds (while we recall that in \cite{LabaPramanik2011} the Cantor sets must have dimension $>2/3$). Our constructions are still random at all scales, but fall into a widely studied class of random measures and sets closely related to the well known fractal percolation model. The method can be easily extended to other random models. Moreover, we are able to simplify various aspects of the rather involved original proof of Theorem \ref{thm:LabaPramanik}. Before stating our main result, Theorem \ref{thm:maximal}, we introduce the random model it involves.

\subsection{Ahlfors regular random sets and measures}\label{sec:nu_def}

Let  $\mathcal{D}_n$ denote the level $n$ dyadic intervals of $\R$:
\[\mathcal{D}_n=\{[k2^{-n},(k+1)2^{-n}]\,:\,k\in\Z\}\,.\]

Let $0<s<1$ and let $a_n\in\{1,2\}$ such that
\[
2^{sn-1}<\beta_n\le 2^{sn} \text{ for all } n\in\N, \quad\text{where } \beta_n=\prod_{i=1}^n a_n.
\]
Starting with the interval $A_0=[1,2]$, we inductively construct random sets $A_n$ as follows. If $a_n=2$, set $A_{n+1}=A_n$. Otherwise, if $a_n=1$, choose, for each $I\in\mathcal{D}_n$ such that $I\subset A_n$, one of the $2$ dyadic sub-intervals of $I$, with all choices being uniform and independent of each other and the previous steps. Let $A_{n+1}$ be the union of the chosen $I\in\mathcal{D}_n$. Then $\{ A_n\}$ is a decreasing sequence of nonempty sets (each $A_n$ consists of $\beta_n$ pairwise disjoint dyadic intervals of length $2^{-n}$), and we set
\[
A=\bigcap_{n=1}^\infty A_n\,.
\]
Let us further define
\begin{equation} \label{eq:def-nu-n}
\nu_n=\frac{2^n}{\beta_n}\mathbf{1}[A_n],
\end{equation}
and note that
\begin{equation}\label{eq:trivial_up}
\nu_n\le 2 \cdot 2^{n(1-s)}
\end{equation}
and that $\nu_n$ is zero off a set of Lebesgue measure at most $2^{n(1-s)}$.

It is easy to check (see e.g. \cite[Proposition 1.7]{Falconer2014}) that $\nu_n$ converges in the weak$^*$-sense to a Borel probability measure $\nu$ and that $\spt\nu=A$. Moreover, $\nu$ is Ahlfors $s$-regular, that is, there exists a (deterministic) constant $C>0$ such that
\[
C^{-1}\cdot r^s \le \nu(B(x,r)) \le C \cdot r^s
\]
for all $x\in A$ and all $r\in (0,1]$. In particular, the Hausdorff dimension of $A$, and of $\nu$, equal $s$ (deterministically) and
\begin{align}
\label{eq:ar}\nu(I), \nu_n(I)&\le 3|I|^s\text{ for all intervals }I\subset[0,1]\,,
\end{align}
where $|I|$ denotes the length of the interval. Note that in the above notation, and also in what follows, we will often identify the functions $\nu_n$ with the measure $d\nu_n(x)=\nu_n dx$. This also applies to Cartesian powers of $\nu_n$.

\subsection{Main result}
\label{subsec:mainresult}

We can now state our main result:

\begin{thm}\label{thm:maximal}
Let $s\in (1/2,1)$ and let $d=2\lceil \tfrac{1}{2-2s}-1\rceil$. Further, define
\begin{align*}
\theta_0 &= \frac{ds+1-d}{2}>0,\\
\xi_0 &=d+1-(d+2)s\ge 0
\end{align*}
If $\nu$ is the random measure defined above, then almost surely
$M_\nu$ is bounded  from $L^p(\R)$ to $L^q(\R)$ for all
\[
p > p_0:=
\frac{(2+d)\theta_0+d\xi_0}{(1+d)\theta_0+(d-1)\xi_0}
%
\]
and
\[
p\le q \le \frac{p}{p_0-1}.
\]
\end{thm}
See Figure \ref{fig:plot} for an illustration. We make some remarks on the statement of theorem and its proof.

\begin{figure}[H]
	\includegraphics[width=0.9\textwidth]{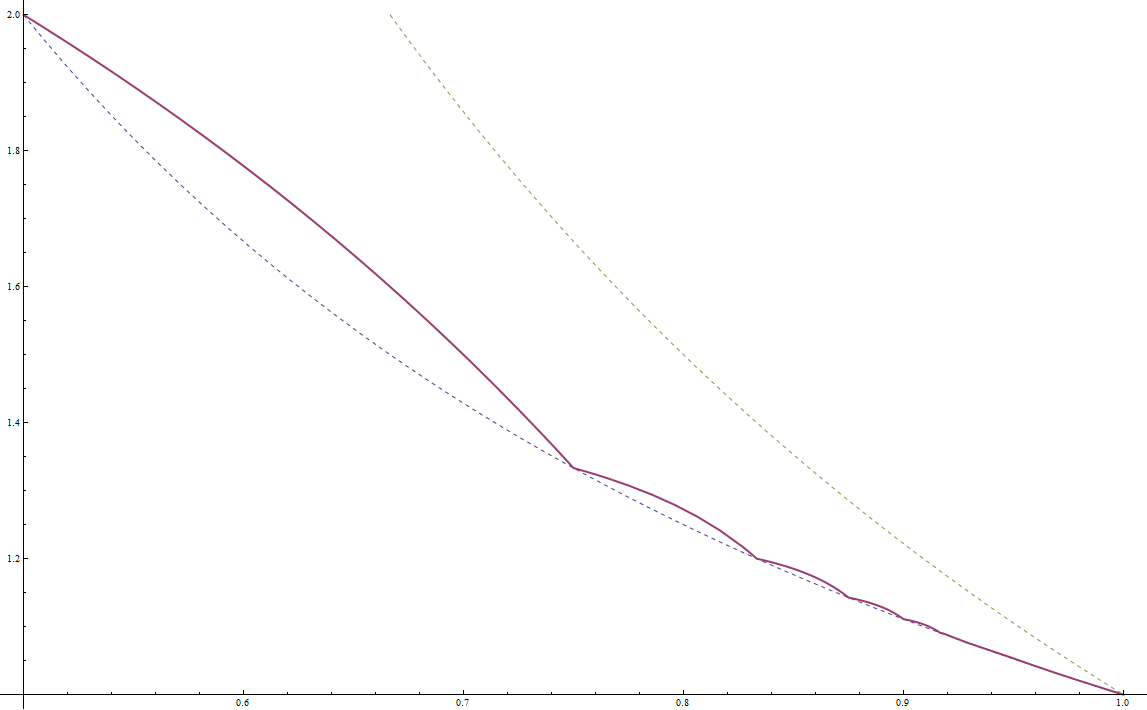}
	\caption{The solid curve shows the value of $p_0=p_0(s)$ provided by Theorem \ref{thm:maximal} for each $s\in (1/2,1)$, so that $M_\nu$ is a.s. bounded on $L^p$ for $p>p_0$. The lower dashed curve is $1/s$; by Lemma \ref{lem:sharpness}, the maximal operator can't be bounded on $L^p$ for $p<1/s$. The upper dashed curve is the value $p_0=(2-s)/s$, $s>2/3$, given by the theorem of {\L}aba-Pramanik (Theorem \ref{thm:LabaPramanik}) }
	\label{fig:plot}
\end{figure}

\begin{rem}
As remarked in Lemma \ref{lem:sharpness}, $M_\nu$ is not bounded from $L^p$ to any $L^q$ if $p<1/s$. In the special case $s=1-\frac{1}{2d'}$, $d'\ge 2$, the above theorem shows that $M_\nu$ is bounded on $L^p$ for any $p>1/s$. So Theorem \ref{thm:maximal} provides the first instance of a sharp bound for maximal operators associated to Cantor sets of fractional dimension (up to the endpoint). The lower bound $p_0$ provided by Theorem \ref{thm:maximal} is continuous in $s$, so the theorem also provides nearly sharp bounds for $s$ close to $1-\frac{1}{2d'}$ for $d'\ge 2$, and also for $s$ slightly larger than $1/2$ (but it provides no information for $s\le 1/2$).
\end{rem}

\begin{rem}
Unlike the construction of \cite{LabaPramanik2011}, the measure $\nu$ is Ahlfors-regular. This may be seen as a desirable geometrical property that holds in all classical examples such as smooth manifolds. Additionally, as a direct consequence of \cite[Theorem 14.1]{ShmerkinSuomala2018}, the measure $\nu$ is almost surely a Salem measure - again as is the case for manifolds for non-zero Gaussian curvature.
\end{rem}

\begin{rem}
While we have chosen the model described in Section \ref{sec:nu_def} for concreteness, the argument of the proof extends to a wide class of \emph{subdivision fractals} - see \cite[\S 5.2]{ShmerkinSuomala2018}.
\end{rem}

\begin{rem}
For proving Theorem \ref{thm:maximal}, we follow \cite{LabaPramanik2011} (with some minor simplifications) to reduce the claim to a purely geometric fact regarding intersections of random sets with lines, see Theorem \ref{thm:main} below. Our proof differs from that of \cite{LabaPramanik2011} in the probabilistic argument to establish the intersection result; it is here that our argument is more effective, and perhaps also simpler. It is inspired by our earlier work \cite{ShmerkinSuomala2018, ShmerkinSuomala2020}. A new aspect is that we also need to consider intersections with lines which are not ``transversal'' and this requires a more delicate analysis.
\end{rem}

We mention some direct  applications of Theorem \ref{thm:maximal}.  For the first two we follow \cite{LabaPramanik2011}. We fix $s\in (1/2,1)$ and the threshold $p_0$ provided by the theorem.

\begin{enumerate}
  \item Theorem \ref{thm:maximal} implies the boundedness of the \emph{unrestricted} maximal operator
\[
\overline{M}_\nu f(x) = \sup_{r>0} \int f(x+r y)\,d\nu(y)
\]
for $p=q>p_0$, by the argument in \cite[Section 7]{LabaPramanik2011} which is very general and does not rely on the specific random construction in that paper. For $q>p$ a similar result holds but one needs to weigh the unrestricted operator for scaling reasons, see \cite[Eq. (1.4)]{LabaPramanik2011}.

  \item Maximal operator bounds imply differentiation results. In particular, in the context of Theorem \ref{thm:maximal}, if $f\in L^p(\R)$ with $p>p_0$, then
\[
f(x) = \lim_{r\to 0} \int f(x+ry)\,d\nu(y)
\]
for a.e. $x\in\R$. See \cite[\S 8.1]{LabaPramanik2011}.
  \item Finally, maximal operator bounds also yield information about the Lebesgue measure of the union of families of Cantor sets. Namely, if the Borel set $E\subset\R$ contains a set of the form $x+ r A$ (for some $r$ depending on $x$) for each $x$ in a set $F$ of positive Lebesgue measure, then $E$ has positive Lebesgue measure. Indeed, $M_\nu \mathbf{1}_E(x) =1$ for all $x\in F$, which immediately yields $|E|>0$ (for this it is enough to have a bound $\| M_\nu f\|_q\le C\|f\|_p$ for any $p<\infty$). Similar problems for other classes of Cantor sets have been considered in \cite{Hochman2017, Laba2018}; the Cantor sets in those papers are (essentially) deterministic and the conclusions are therefore much weaker.
\end{enumerate}

\section{Reduction of the main theorem to intersection estimates}

\subsection{From intersection estimates to maximal operator bounds}

We will derive our main theorem by analyzing the intersections of the self-products of $\nu$ with lines. To reduce the main theorem to such intersection estimates,  we will use the following theorem that can be inferred from the framework of \cite{LabaPramanik2011}.

\begin{thm}\label{thm:LabaPramanikAbstract}
For $n\in\N$, let $A_n\subset [1,2]$ be a finite union of  closed intervals such that $A_{n+1}\subset A_n$. Denote
$\nu_n=|A_n|^{-1}\mathbf{1}[A_n]$ and $\sigma_n=\nu_{n+1}-\nu_n$. Suppose that $\nu_n$ converges in the weak*-sense to a measure $\nu$. Furthermore assume that $d\ge 2$ is an even integer and there are constants $\theta,\xi,K>0$, such that for all measurable choices of $r\colon[-4,0]\to[1,2]$ and for all $\Omega\subset [0,1]$ it holds that
\begin{align}
|\Omega|^{1-d}\int_{\Omega^d}\int\prod_{j=1}^d\sigma_k\left(\frac{z-x_j}{r(x_j)}\right)\,dz\,dx &\le K\exp(-\theta k)\,,\label{eka}\\
|\Omega|^{-1-d}\int_{\Omega^{d+2}}\int\prod_{j=1}^{d+2}\sigma_k\left(\frac{z-x_j}{r(x_j)}\right)\,dz\,dx &\le K\exp(\xi k)\,.\label{toka}
\end{align}
Then $M_\nu$ is bounded from $L^p(\R)$ to $L^q(\R)$, whenever
\[
p> p_0 := 
\frac{(2+d)\theta+d\xi}{(1+d)\theta+(d-1)\xi}
%
\]
and
\[
p \le q \le \frac{p}{p_0-1}.
\]
\end{thm}

Since this is not stated in this form in \cite{LabaPramanik2011}, in the rest of this section we discuss the main steps of the proof, referring to \cite{LabaPramanik2011} for most details. We point out two differences between our approach and that of {\L}aba and Pramanik that are not essential but help us simplify parts of the argument. The first is that we don't have an explicit split between ``internal'' and ``transverse'' intersections. These concepts from \cite{LabaPramanik2011} were inspired by Bourgain's proof of the boundedness of the circular maximal operator, but such a dichotomy is not needed in our approach. The second is that they discretize the family of measurable functions $r$ at each scale at the ``deterministic'' stage, while we perform a similar discretization in the probabilistic part of the argument.

Roughly speaking, Theorem \ref{thm:LabaPramanikAbstract} follows a classical scheme involving discretization, linearization, dualization and interpolation arguments, although extra care is required at some steps. Since $\nu$ is fixed, we denote $M=M_\nu$ for simplicity. Firstly, using that $q\ge p$ it is enough to show that $M$ is bounded from $L^p[0,1]$ to $L^q(\R)$; this is due to the fact that we are dealing with the restricted operator ($r\in [1,2]$); see the proof of \cite[Lemma 3.1]{LabaPramanik2011} for details. Next, for $f\in C[0,1]$ define
\[
M_kf(x) = \sup_{r\in [1,2]} \left|  f(x+ry) \sigma_k(y) \, dy \right|.
\]
It is an easy consequence of the weak$^*$-convergence of $\nu_n$ to $\nu$ that
\[
M f \le N f + \sum_{k=1}^\infty M_k|f|, \quad \text{where } Nf(x)=\sup_{r\in [1,2]} \int |f(x+ry)|d\nu_1(y).
\]
As $\nu_1$ is dominated by a bounded multiple of Lebesgue measure, this reduces the problem to the study of the operators $M_k$.

Next, we linearize the problem. It is easy to see that for $f\in C[0,1]$,
\[
\|M_k f(x)\|_{L^q(\R)} \le 4 \sup_{r\colon[-4,0]\to[1,2] \text { measurable}} \|M_{r,k} f(x)\|_{L^q(\R)},
\]
where $M_{k,r}$ is the (linear) operator
\[
M_{k,r}f(x)= \int f(z) \sigma_k\left(\frac{z-x}{r(x)}\right)\,d z.
\]
See \cite[Proposition 3.2]{LabaPramanik2011} for details. (It is enough to consider $x\in [-4,0]$ since $r(x)\in [1,2]$ and $\sigma_k$ is supported on $[1,2]$.) Thus the claim will follow if
\begin{equation} \label{eq:M-k-r-bound}
\|M_{k,r}\|_{L^p[0,1]\to L^q[-4,0]}  \le \delta_k
\end{equation}
for some summable sequence $(\delta_k)$ independent of the choice of $r$. The adjoint operator to $M_{k,r}$ is
\[
M^*_{k,r}g(z) = \int g(x) \sigma_k\left(\frac{z-x}{r(x)}\right)\,d x.
\]
Using interpolation and duality, one can see that if $q_0\ge 2$ and the restricted bound
\begin{equation} \label{eq:restricted-exponential-decay}
\|M^*_{k,r}\mathbf{1}[\Omega]\|_{L^{q_0}[-4,0]} \le C 2^{-\zeta k} |\Omega|^{\tfrac{q_0-1}{q_0}}
\end{equation}
holds for all $\Omega \subset [0,1]$ and some constants $C,\zeta>0$, then \eqref{eq:M-k-r-bound} holds for
\[
p>\frac{q_0}{q_0-1}, \quad q=(q_0-1)p\,,
\]
for an exponentially decaying sequence $(\delta_k)$ (depending on $p,q$). Then it also holds for $q\in [p, (q_0-1)p]$. See \cite[Lemma 3.4]{LabaPramanik2011} for the details of this step, that requires special care. (This is where we use the hypothesis on $q$.)

On the other hand, it is easy to check (see \cite[Proof of Prop. 4.2]{LabaPramanik2011}) that if $d\ge 2$ is an even integer, then
\begin{equation} \label{eq:dual-integer-d}
\|M^*_{k,r}\mathbf{1}[\Omega]\|_{L^d[-4,0]}^d =\int_{\Omega^d}\int\prod_{j=1}^d\sigma_k\left(\frac{z-x_j}{r(x_j)}\right)\,dz\, dx.
\end{equation}
Let $u\in (0,1]$ satisfy $u>\tfrac{d\xi}{(d+2)\theta+d\xi}$ (where $\theta,\xi$  are as in Theorem \ref{thm:LabaPramanikAbstract}). Let $q_u\in (d,d+2]$ satisfy
\[
\frac{1}{q_u}=\frac{u}{d}+\frac{1-u}{d+2}
\]
so that $q_u=\tfrac{d(d+2)}{2u+d}$ (and $\tfrac{q_u}{q_u-1}=\tfrac{d(d+2)}{d(d+1)-2u}$).
Using H\"older's inequality,
\begin{align*}
\|M^*_{k,r}\mathbf{1}[\Omega]\|_{q_u}\le\|M^*_{k,r}\mathbf{1}[\Omega]\|_{d}^u\,\|M^*_{k,r}\mathbf{1}[\Omega]\|_{d+2}^{1-u}.
\end{align*}
Recalling the assumptions \eqref{eka}--\eqref{toka}, we arrive at
\begin{align*}
\|M^*_{k,r}\mathbf{1}[\Omega]\|_{q_u}&\le K'|\Omega|^{\tfrac{u(d-1)}{d}+\tfrac{(1-u)(d+1)}{(d+2)}}\exp\left(k\left(\frac{-u\theta}{d}+\frac{(1-u)\xi}{d+2}\right)\right)\\
&=K'|\Omega|^{\tfrac{q_u-1}{q_u}}\exp(-\zeta k)\,,
\end{align*}
for some $\zeta>0$ provided $u>\tfrac{d\xi}{(d+2)\theta+d\xi}$.

After some algebra, letting $u\downarrow \tfrac{d\xi}{(d+2)\theta+\xi}$  completes the proof of Theorem  \ref{thm:LabaPramanikAbstract}.

\subsection{Reduction of the main result to probabilistic intersection estimates}

Let us continue to denote
\begin{equation}\label{eq:sigma_def}
\sigma_n=\nu_{n+1}-\nu_n,
\end{equation}
where $\nu_n$ is as in \eqref{eq:def-nu-n}.

Using Theorem \ref{thm:LabaPramanikAbstract}, the following probabilistic estimate will easily imply our main result, Theorem \ref{thm:maximal}.
\begin{thm}\label{thm:main}
Let $\sigma_k$ be as in \eqref{eq:sigma_def} and fix $d\ge 2$. Let
\begin{equation} \label{eq:main}
\Phi_k^{(d)} = \sup_{r\colon[-4,0]\to[1,2], \Omega\subset[0,1]}|\Omega|^{1-d}\int_{\Omega^d}\int\prod_{j=1}^d \sigma_k\left(\frac{z-x_j}{r(x_j)}\right)\,dz\,dx,
\end{equation}
where the supremum runs over all measurable functions $r$. Note that $\Phi_k^{(d)}$ is a random variable.

Then for all $k\in\N$:
\begin{enumerate}
  \item \label{it:i} If $s>1-\tfrac1d$ and $0<\theta<(ds+1-d)/2$, then $\Phi_k^{(d)}\le K\exp(-\theta k)$,
  \item \label{it:ii} If $s\le 1-\tfrac1d$ and $\xi>d-1-ds \ge 0$ , then $\Phi_k^{(d)}\le K\exp(\xi k)$,
\end{enumerate}
where $K$ is an almost surely finite random variable (depending on $d,s,\theta,\xi$).
\end{thm}

We will prove Theorem \ref{thm:main} in Section \ref{sec:proof-of-main-thm}. Theorem \ref{thm:maximal} is a simple consequence of Theorems \ref{thm:LabaPramanikAbstract} and \ref{thm:main}. Indeed, given $s>1/2$, then $d=2\lceil \tfrac{1}{2-2s}-1\rceil$ 
satisfies $1-\tfrac{1}{d}<s\le 1-\tfrac{1}{d+2}$. Theorem \ref{thm:main} ensures that if $0<\theta<(ds+1-d)/2$ and $\xi>d+1-(d+2)s$, then the hypotheses of Theorem \ref{thm:LabaPramanikAbstract} are satisfied. Theorem \ref{thm:maximal} follows by letting $\theta\uparrow  \theta_0=(ds+1-d)/2$ and $\xi\downarrow \xi_0= d+1-(d+2)s$.

\begin{rem}
Theorem \ref{thm:main} holds for all integers $d\ge 2$, but when we apply it to prove Theorem \ref{thm:maximal}, we restrict to even values of $d$ to be able to use Theorem \ref{thm:LabaPramanikAbstract}. The assumption that $d$ is even is essential in the proof of Theorem \ref{thm:LabaPramanikAbstract}, since \eqref{eq:dual-integer-d} fails for $d$ odd.
\end{rem}

\section{Proof of Theorem \ref{thm:main}}

\label{sec:proof-of-main-thm}

\subsection{Outline of proof}

In this section we prove Theorem \ref{thm:main} which, as explained above, implies Theorem \ref{thm:maximal}. We start by giving an outline of the proof. It is easy to recast the inner integral in the definition of $\Phi_k^{(d)}$ given in \eqref{eq:main} as an integral of the $d$-fold cartesian product  $\lambda_n:=\sigma_n\times \cdots \times \sigma_n$ over a line $L_{x,r}$ determined by the $x_j$ and the choice of the function $r$. Because the functions $\sigma_k$ are highly singular, the inner integral blows up on the diagonals $x_i=x_j$. On the other hand, if the $x_i$ are well separated (essentially what is called a ``transverse intersection'' in \cite{LabaPramanik2011}) one would expect a large amount of cancellation. Indeed, a stochastic induction in $n$ coupled with Hoeffding-type estimates can be used to show that, for the fixed line $L_{x,r}$ the random variable
\[
X_n(x,r) = \int_{L_{x,r}} \lambda_n  \, d\mathcal{H}^1.
\]
decays exponentially in $n$, with overwhelming probability. Here the line $L_{x,r}$ is fixed, but because the probability that $X_n(x,r)$ does \emph{not} decay exponentially is so small (sub-exponential), it follows tht $X_n(x_j,r_j)$ decay exponentially, with a uniform bound, for any collection $L_{x_j,r_j}$ of transversal lines of size exponential in $n$. By choosing the $(x_j,r_j)$ densely in the parameter space, a deterministic continuity bound in $x$ and $r$ can then be used to extend the estimate to \emph{all} transversal lines. So far, this scheme is similar to that of \cite{ShmerkinSuomala2020}. However, we need to deal also with non-transversal lines, that is lines for which $|x_i-x_j|$ is small for some $i\neq j$ (roughly the ``internal tangencies'' of \cite{LabaPramanik2011}). For such lines $X_n(x,r)$ will be much larger, but on the other hand their weight in the integral in \eqref{eq:main} is small. If $\min_{i\neq j}|x_i-x_j|\sim 2^{-m}$, then we have essentially no control on $X_n$ for $n<m$ due to the lack of independence between the coordinates $i$ and $j$ of $\lambda_n$, but a deterministic bound can easily be given. However, for $n>m$ we regain independence and are thus able to run the stochastic induction and achieve a bound on $X_n(x,r)$ with overwhelming probability that, while it increases with $m$, is still small enough that the desired bounds on $\Phi_k^{(d)}$ can be achieved.

\subsection{Notation and setup}

We introduce some notation to be used throughout the proof. We will denote by $C,c$ positive and finite constants whose precise value is of no importance and may change even inside a given chain of inequalities. When necessary to specify a constant inside a proof, we will use subscripts such as $C_1,c_2$. We also use the notation $A\lesssim B$ instead of $A\le C B$.

We denote $\Delta=\{x\in\R^d\,:\,x_i=x_j\text{ for some }i\neq j\}$ and by $E(\delta)$, the $\delta$-neighbourhood of a set $E$ so that  $E(\delta)=\{x\in\R^d\,:\,d(x,E)\le \delta\}$. Recall that $\mathcal{D}_n$ denotes level $n$ dyadic intervals of $\R$, and let
$\mathcal{Q}_n$ be their $n$-dimensional counterparts:
\[
\mathcal{Q}_n=\{D_1\times\ldots\times D_d\,:\,D_i\in\mathcal{D}_n\}\,.
\]
By $|\cdot|$, we denote the Lebesgue measure on $\R$ and $\R^d$. One-dimensional Hausdorff measures is denoted by $\mathcal{H}^1$. We denote $[n]=\{1,2,\ldots,n\}$.

We work with the random construction specified in Section \ref{sec:nu_def}. Recall that $\sigma_n=\nu_{n+1}-\nu_n$. Given $d\in\N_{\ge 2}$ (which is from now on fixed), let  $\lambda_n$ be the $d$-fold Cartesian power $\lambda_n=\sigma_n\times\ldots\times\sigma_n$. Furthermore, let $\mu_n=\nu_n\times\cdots\times\nu_n$.

Given $x,z\in\R$ and $r\in[1,2]$, write
\[
\phi(x,r,z)=\frac{z-x}{r}\,.
\]
For each $x\in[-4,0]^d$, $r\in[1,2]^d$, let $L_{x,r}$ denote the line
\[
L_{x,r}=\left\{(\phi(x_1,r_1,z),\ldots,\phi(x_d,r_d,z))\,:\,z\in\R\right\}\subset\R^d,
\]
and note that (since $r_i\ge 1$),
\begin{equation}\label{eq:Fub}
\int\prod_{i=1}^d\sigma_n(\phi(x_i,r_i,z))\,dz\le\int_{L_{x,r}}\lambda_n\,d\mathcal{H}^1\,.
\end{equation}
We define the random variables
\begin{align} \label{eq:def-X-n}
X_n(x,r)&=X_n(x_1,\ldots,x_d,r_1,\ldots,r_d)=\int_{L_{x,r}}\lambda_n\,d\mathcal{H}^1\,,
\end{align}
for $x=(x_1,\ldots,x_d)\in[-4,0]^d$ and $r=(r_1,\ldots,r_d)\in[1,2]^d$.

\subsection{The key lemmas}

We denote
 \begin{align*}
\Gamma_m&=\{x\in[-4,0]^d\,:\,4\cdot 2^{-m}<\min|x_i-x_j|\le 4\cdot 2^{1-m}\}\,,\\
\widetilde{\Gamma}_m&=\{x\in[-4,0]^d\,:\,\min|x_i-x_j|\le 4\cdot 2^{-m}\}\,.
\end{align*}

The core of the proof lies in the following (closely related) probabilistic lemmas:
\begin{lemma}\label{lem:C_k_small_gen}
Fix $m\in\N$, $1-\tfrac1d<s<1$, $0<\theta<(ds+1-d)/2$ and $0<\delta<1-d+ds-2\theta$. Then there are deterministic constants $C_1, c_2>0$ (depending on $\theta$, $\delta$ and $s$) such that the following holds.

Given $n\ge m$, consider the event $E_{m,n}$ defined as
\begin{equation} \label{eq:def-E-m-n}
\sup_{x\in \Gamma_m, r\in [1,2]^d}X_n(x,r)>C_1 \left( 2^{ms+n(d-1-ds)}+  2^{m(d-1)(1-s)/2}2^{-\theta n}\right).
\end{equation}

Then
\[
\PP\left(\bigcup_{n=m}^\infty E_{m,n} \right)\le C_1 \exp(-c_2 2^{\delta m}).
\]
\end{lemma}

\begin{lemma}\label{lem:no_intersections}
Fix $m\in\N$ and $s\le 1-\tfrac1d$. Then for any $\xi>d-1-ds\ge 0$ there are $C_1, c_2>0$ such that
\[
\PP\left(X_n(x,r)> C_1 2^{ms}2^{n\xi} \text{ for some } n\ge m, x\in\Gamma_m, r\in [1,2]^d\right)
\le C_1 \exp(-c_2  2^{\delta m}),
\]
where $\delta=\xi-(d-1-ds)>0$.
\end{lemma}

 We recall that although in Theorem \ref{thm:main} the parameter $r\in[1,2]^d$ is a function of $x$ of the form $r=(\widetilde{r}(x_1),\ldots,\widetilde{r}(x_d))$ for some measurable $\widetilde{r}\colon[-4,0]\to[1,2]$, in Lemmas \ref{lem:C_k_small_gen} and \ref{lem:no_intersections} there is absolutely no relation between $x$ and $r$. In fact, this is the reason why they are useful for obtaining a bound for $\Phi_k^{(d)}$ (defined in \eqref{eq:main}) that holds irrespective of the choice of $\widetilde{r}$. Before proving the lemmas, let us show how they imply Theorem \ref{thm:main}.

\begin{proof}[Proof of Theorem \ref{thm:main} (assuming Lemmas \ref{lem:C_k_small_gen} and \ref{lem:no_intersections})]
Since $s$ and $d$ are fixed throughout the proof, all implicit constants are allowed to depend on them.

Consider first the case $1-1/d<s<1$, so that $1-d+ds>0$ by assumption, and fix
\begin{equation} \label{eq:assumption-theta}
0<\theta<(1-d+ds)/2.
\end{equation}

Fix $\Omega\subset[0,1]$ and a measurable function $r:[-4,0]^d\to [1,2]^d$. Recalling \eqref{eq:Fub} and \eqref{eq:def-X-n}, we have
\begin{equation}\label{eq:split}
\begin{split}
\int_{\Omega^d}&\int\prod_{i=1}^d\sigma_n(\phi(x_i,r_i(x),z))\,dz\,dx\\
&\le \int_{\Omega^d\cap\widetilde{\Gamma}_n} X_n(x,r(x))\,dx + \sum_{m=1}^n\int_{\Omega^d\cap\Gamma_m} X_n(x,r(x))\,dx\,.
\end{split}
\end{equation}

To analyze the first integral, fix $x\in\widetilde{\Gamma}_n$ and $r\in[1,2]^d$.
From the estimate $|A_n|\le 2^{-n}\beta_n\le 2^{(s-1)n}$, we observe that  $(\nu_n)^d$ (and thus also $\lambda_n$) vanishes off a set of $\mathcal{H}^1$ measure $\lesssim 2^{(s-1)n}$ on the line $L_{x,r}$. Making use of the trivial bound $\lambda_n\lesssim 2^{(1-s)dn}$ (recall \eqref{eq:trivial_up}), we thus have the deterministic estimate
\begin{equation*}
X_n(x,r)\lesssim 2^{(1-s)(d-1)n}\text{ for all }x\in\widetilde{\Gamma}_n, r\in[1,2]^d\,.
\end{equation*}
Using Fubini's theorem,
\[|\Omega^d\cap\widetilde{\Gamma}_n|\lesssim |\Omega|^{d-1}2^{-n}\,.\]
Combining the last two estimates,
\begin{equation}\label{eq:Gamma-tilde}
\int_{\Omega^d\cap\widetilde{\Gamma}_n} X_n(x,r(x))\,dx\lesssim  |\Omega|^{d-1}2^{n(d-2+(1-d)s)} \le |\Omega|^{d-1}2^{-\theta n}\,.
\end{equation}
using that $d-2+(1-d)s<d-1-ds<-\theta$, where the second inequality follows from \eqref{eq:assumption-theta}.

To estimate $X_n(x,r)$ on $\Gamma_m$, $m\le n$, we apply Lemma \ref{lem:C_k_small_gen}. Summing over all $m\in\N$ in the lemma and using the Borel-Cantelli Lemma, we see that almost surely, there exists $K<\infty$, such that
\begin{equation*}
X_n(x,r)\le K \left(2^{ms+n(d-1-ds)}+  2^{m(d-1)(1-s)/2}2^{-\theta n}\right)
\end{equation*}
for all $n\ge m\ge 1$, $x\in\Gamma_m$ and $r\in[1,2]^d$ (we absorb the deterministic constant $C_1$ into $K$ for convenience). On the other hand, Fubini's theorem gives
\[
|\Omega^d\cap\Gamma_m|\lesssim |\Omega|^{d-1}2^{-m}\,.
\]
We thus have
\[
\int_{\Omega^d\cap\Gamma_m}X_n(x,r(x))\,dx \lesssim K |\Omega|^{d-1}\left(2^{m(s-1)+n(d-1-ds)}+2^{m(-1+(d-1)(1-s)/2)-\theta n}\right).
\]
Using $s>1-1/d$ and \eqref{eq:assumption-theta}, we see that
\[
\sum_{m=1}^n 2^{m(s-1)+n(d-1-ds)} \lesssim 2^{n(d-1-ds)} \le 2^{-\theta n},
\]
\[
\sum_{m=1}^n 2^{m(-1+(d-1)(1-s)/2)-\theta n} \lesssim 2^{-\theta n}.
\]
Summing over all $m=1,\ldots,n$, we get
\begin{equation}  \label{eq:Gamma-sum}
\sum_{m=1}^n \int_{\Omega^d\cap\Gamma_m}X_n(x,r(x))\,dx \lesssim K |\Omega|^{d-1} 2^{-\theta n}.
\end{equation}
Combining \eqref{eq:split}, \eqref{eq:Gamma-tilde} and \eqref{eq:Gamma-sum} yields the claim in the case $s>1-1/d$.

Consider now the case $s\le 1-1/d$; the proof is very similar (in fact, simpler), except that we rely on Lemma \ref{lem:no_intersections} instead. By Lemma \ref{lem:no_intersections} and Borel-Cantelli, there is a finite random variable $K$ such that
\[
X_n(x,r) \le K \, 2^{ms} 2^{n\xi}
\]
for all $x\in\Gamma_m$, $r\in [1,2]^d$ and $n \ge m$. Combining this with Lemma \ref{lem:diag_transversality} and Fubini's theorem as in the previous part, we conclude that
\begin{align*}
\int_{\Omega^d}&\int\prod_{i=1}^d\sigma_n(\phi(x_i,r(x_i),z))\,dz\,dx\\
&\le\int_{\Omega^d\cap\widetilde{\Gamma}_n}\ldots+\sum_{m=1}^n\int_{\Omega^d\cap\Gamma_m}\ldots\\
&\lesssim |\Omega|^{d-1}2^{n(d-2+(1-d)s)}+K|\Omega|^{d-1} \sum_{m=1}^n 2^{m(s-1)}2^{n\xi}\\
&\lesssim K|\Omega|^{d-1}2^{n\xi}\,.
\end{align*}

\end{proof}

\subsection{Proof of Lemmas \ref{lem:C_k_small_gen} and \ref{lem:no_intersections}}

It remains to prove lemmas \ref{lem:C_k_small_gen} and \ref{lem:no_intersections}. We continue to think of $s,d$ as constants and hence all implicit constants are allowed to depend on them (but not on $m,k,r,x$!) In addition to $X_n(x,r)$, we will consider the random variables
\begin{align} \label{eq:def-Y-n}
Y_n(x,r)=\int_{L_{x,r}}\mu_n\,d\mathcal{H}^1\,,
\end{align}
and
\begin{align}
Z_n(x,r)=\int_{L_{x,r}}(\mu_{n+1}-\mu_n)\,d\mathcal{H}^1=Y_{n+1}(x,r)-Y_n(x,r)\,.\label{eq:def-Z-n}
\end{align}
 We begin by stating some deterministic elementary bounds, i.e. they hold for all possible choices of the sets $A_n$.

\begin{lemma}\label{lem:Y_k_triv_bound}
For all $x\in [-4,0]^d$, $r\in[1,2]^d$,
\[
Y_m(x,r)\lesssim 2^{m(d-1)(1-s)}\,.
\]
\end{lemma}

\begin{proof}
The claim follows directly from the bounds
\[
\mathcal{H}^1(L_{x,r}\cap (A_m)^d) \lesssim \mathcal{H}^1(A_m) \le  2^{-m}\beta_m\le  2^{m(s-1)},
\]
and $\mu_m\lesssim 2^{md(1-s)}$.
\end{proof}

Recall that $\Delta$ stands for the union of the diagonals $\cup_{i\neq j}\{ (x_1,\ldots,x_d): x_i =x_j\}$, and $\Delta(r)$ is the $r$-neighborhood of $\Delta$.
\begin{lemma}\label{lem:diag_transversality}
For all $x\in\Gamma_m$, $r\in[1,2]^d$ and $k\ge m$,
\begin{align*}
\int_{L_{x,r}\cap\Delta(2^{-k})}\mu_k\, d\mathcal{H}^1\lesssim 2^{ms+k(d-1-ds)}\,,\\\int_{L_{x,r}\cap\Delta(2^{-k})}\lambda_k\, d\mathcal{H}^1\lesssim 2^{ms+k(d-1-ds)}\,.
\end{align*}
\end{lemma}

\begin{proof}
To start, we claim that for all $x\in\Gamma_m$, $r\in[1,2]^d$ and each $k\ge m$, we can cover $L_{x,r}\cap\Delta(2^{-k})$ by the union of $\lesssim 1$ intervals each of length $\lesssim 2^{m-k}$.  To verify this, let $i,j\in[d]$, $i\neq j$. Since there are $d(d-1)/2$ such pairs, it sufffices to show that
\begin{equation*}
\left|\Delta_{i,j,k}\right|\lesssim  2^{m-k}\,,
\end{equation*}
where
$\Delta_{i,j,k}=\{z\in [-3,4]\,:\,\left|(z-x_i)/r_i-(z-x_j)/r_j\right|<2^{-k}\}$.
Recall that $[-3,4]$ is the relevant range for $z$ since  $\phi(x,r,z)\notin[1,2]$ outside this interval.
 Without loss of generality, consider the case $i=1$, $j=2$.
If $|r_1-r_2|<\tfrac1{16}2^{-m}$, then for all $z\in[-3,4]$,
\begin{align*}
&\left|\frac{z-x_1}{r_1}-\frac{z-x_2}{r_2}\right|\ge\left|\frac{x_1-x_2}{r_1}\right|-\left|\frac{(r_1-r_2)x_2}{r_1r_2}\right|-|z|\left|\frac{r_2-r_1}{r_1r_2}\right|\\
&\ge\frac{|x_1-x_2|}2-8|r_2-r_1|
> 2^{1-m}-2^{-m-1}
>2^{-m}\,,
\end{align*}
and so $\Delta_{1,2,k}$ is empty for all $k\ge m$. If $|r_1-r_2|\ge\tfrac1{16}2^{-m}$, the line $L_{x,r}$ makes an angle $\gtrsim 2^{-m}$ with the plane $\{x_1=x_2\}$ implying that $|\Delta_{1,2,k}|\lesssim 2^{m-k}$.

Now, if $J\subset[1,2]$ is an interval of length $|J|\lesssim 2^{m-k}$,  we conclude from \eqref{eq:ar} that
\[
|A_k\cap J|=\beta_k 2^{-k}\nu_k(A_k\cap J)\lesssim 2^{sm-k}
\]
and thus also $|(A_k)^d\cap L_{x,r}\cap J|\lesssim 2^{sm-k}$, whenever $J$ is such an interval on $L_{x,r}$. Recalling \eqref{eq:trivial_up}, this yields the claim.
\end{proof}

We will also make use of the following Hoeffding--Janson inequality, see \cite[Theorem 2.1]{Janson2004}. We recall some terminology. Consider a family of random variables, $\{X_i\}_{i\in \mathbf{I}}$, indexed by $\mathbf{I}$. A graph with vertex set $\mathbf{I}$ is called a \emph{dependency graph} for $\{X_i\}_{i\in \mathbf{I}}$ if the random variable $X_i$ is independent from $\{ X_j:j\in J\}$ whenever there is no edge connecting $i$ to $J$ (here $i\in I$ and $J\subset \mathbf{I}$). In our context, a natural dependency graph for the random variables $\mu_{n+1}|_{Q}$, $Q\in\mathcal{Q}_n$, conditional on $A_n$, is obtained by connecting $Q=D_1\times\ldots\times D_d$ and $Q'=D'_1\times\ldots\times D'_d$ via an edge whenever $D_i=D'_j$ for some pair $i,j\in[d]$.

\begin{lemma} \label{lem:HoeffdingJanson}
Let $\{ X_i: i\in \mathbf{I}\}$ be zero mean random variables uniformly bounded by $R>0$, and with a dependency graph whose vertices have degree bounded by $D$. Then
\begin{equation*}
\mathbb{P}\left(\left|\sum_{i\in \mathbf{I}} X_i\right|> a\right) \le 2\exp\left(\frac{- 2 a^2}{(D+1)(\#\mathbf{I}) R^2}\right).
\end{equation*}
\end{lemma}

\begin{proof}[Proof of Lemma \ref{lem:C_k_small_gen}]
Fix $m\in\N$ for the rest of the proof. Given $k\ge m$, let us define an event $\mathcal{G}_k$ as
\[
\sup_{m\le n\le k, x\in\Gamma_m, r\in[1,2]^d}\{X_n(x,r),Z_n(x,r)\}\le C_1\left( 2^{ms+n(d-1-ds)}+  2^{m(d-1)(1-s)/2}2^{-\theta n}\right)\,,
\]
where $C_1$ is chosen so that $2 C_1$ is the implicit constant from Lemma \ref{lem:diag_transversality}.
%
%
%
%
Note that $\mathcal{G}_k$ is determined by $A_k$. Our goal is to show that
\begin{equation} \label{eq:claim-G-k}
1 - \PP(\mathcal{G}_k) \lesssim \exp(-c_2 2^{\delta m}),
\end{equation}
where $\delta$ is as in the statement of the lemma. This implies the lemma since $\cup_{n=m}^k E_{m,n}$ is contained in the complement of $\mathcal{G}_k$.

Let $\mathcal{G}_{m-1}$ be the sure event (so that it holds deterministically). To prove \eqref{eq:claim-G-k}, we will estimate $\PP(\mathcal{G}_{k}\,|\,\mathcal{G}_{k-1})$ for $k\ge m$.

Recall that  $X_n$ is defined by integrating $\lambda_n$, $Y_n$ is defined by integrating $\mu_n$ and $Z_n$ is defined by integrating $\mu_{n+1}-\mu_n$ (each over the line $L_{x,r}$).
Telescoping, we see
that $Y_k(x,r)=Y_{m}(x,r)+\sum_{n=m}^{k-1}Z_n(x,r)$. We deduce from Lemma \ref{lem:Y_k_triv_bound} that, conditional on $\mathcal{G}_{k-1}$,
\begin{equation}\label{eq:Y_k_bound}
\begin{split}
Y_k(x,r)&\lesssim  2^{m(d-1)(1-s)}+\sum_{n=m}^{k-1}\left(2^{ms+n(d-1-ds)}+ 2^{m(d-1)(1-s)/2}2^{-\theta n}\right)\\
&\lesssim 2^{m(d-1)(1-s)}\,.
\end{split}
\end{equation}

For the time being, let us consider a fixed  $(x,r)\in\Gamma_m\times[1,2]^d$. We next pursue to estimate $X_k(x,r)$. We split the line $L_{x,r}$ into two parts: $L_{x,r}\cap\Delta(2^{-k})$ and $L_{x,r}\setminus\Delta(2^{-k})$ and write
\[X_k(x,r)=\int_{L_{x,r}\cap\Delta(2^{-k})}\lambda_k\,d\mathcal{H}^1+\int_{L_{x,r}\setminus\Delta(2^{-k})}\lambda_k\,d\mathcal{H}^1\,.\]
The first integral will be bounded deterministically using Lemma \ref{lem:diag_transversality}.
To bound the remaining term, $\int_{L_{x,r}\cap\Delta(2^{-k})}\lambda_k$, we borrow an argument from \cite[Lemma 4.7]{ShmerkinSuomala2020}. Given $Q\in\mathcal{Q}_k$, denote
\[
X_Q=\int_{Q\cap L_{x,r}\setminus\Delta(2^{-k})}\lambda_k\,d\mathcal{H}^1\,.
\]
Note that $X_Q$ depends on $x,r$ and $k$. We condition on a fixed realization of $A_{k-1}$ such that $\mathcal{G}_{k-1}$ holds. Let
\[
\mathbf{I} = \left\{  Q\in\mathcal{Q}_k: Q\cap L_{x,r}\neq\varnothing \text{ and } Q\cap A_{k-1} \neq\varnothing \right\}.
\]
For $j> k$, set
\[
\mathbf{I}_j=\left\{Q\in\mathbf{I}\,:\,\sqrt{d}\cdot 2^{-j}<|Q\cap L_{x,r}\setminus\Delta(2^{-k})|\le \sqrt{d}\cdot 2^{1-j}\right\}\,.
\]
We will bound the random sum $\sum_{Q\in\mathbf{I}_j}X_Q$ for each $j>k$ through Lemma \ref{lem:HoeffdingJanson}.  We claim:
\begin{enumerate}[(\rm a)]
\item \label{it:a} $\EE(X_Q)=0$ and $|X_Q|\lesssim 2^{kd(1-s)-j}$ for each $Q\in\mathbf{I}_j$.
\item \label{it:b} There is a dependency graph for $\{X_Q\,:\,Q\in\mathbf{I}_j\}$, whose vertices have a degree bounded by a constant $C=C(d)$.
\item \label{it:c} $\#\mathbf{I}_j\lesssim 2^{(m(d-1)-kd)(1-s)+j}$.
\end{enumerate}
To verify the first item, note that since we are conditioning on a realization of $A_{k-1}$, for $y\in [1,2]^d \setminus \Delta(2^{-k})$ the events that $y_j \in A_k$ are independent. By definition of $\nu_k$ we have
\begin{equation}\label{eq:mart}\mathbb{E}(\nu_{k+1}(y_j)|A_{k})=\nu_{k}(y_j)\,,
\end{equation}
so by independence and linearity we see that $\EE(X_Q)=0$. Since
\[
|\lambda_k|=|\left(\nu_{n+1}-\nu_n\right)^d|\lesssim 2^{dk(1-s)},
\]
and $|Q\cap L_{x,r}\setminus\Delta(2^{-k})|\lesssim 2^{-j}$, the claim \eqref{it:a} holds. The second claim concerning the bound on the dependency degrees follows since for any  $Q=D_1\times\ldots\times D_d\in\mathcal{Q}_k$, there can be at most $C(d)$ cubes $Q'=D'_1\times\ldots\times D'_d\in\mathcal{Q}_k$ such that $D_i=D'_j$ for some pair $i,j\in[d]$. Finally, recalling \eqref{eq:Y_k_bound}, we have
\[
\#\mathbf{I}_j \cdot 2^{-j} \cdot 2^{k(1-s)d} \lesssim Y_k(x,r) \lesssim  2^{m(d-1)(1-s)},
\]
from which \eqref{it:c} follows.

Plugging in \eqref{it:a}--\eqref{it:c} into Lemma \ref{lem:HoeffdingJanson}, we obtain
\begin{equation}\label{eq:sumXQ}
\PP\left(\left|\sum_{Q\in\mathbf{I}_j}X_Q\right|>\tfrac{1}{2(j-k)^2}C_1 2^{m(d-1)(1-s)/2-\theta k }\right)\lesssim \exp\left(-c\tfrac{2^{j-k}}{(j-k)^4}2^{k(1-d+ds-2\theta)}\right)\,.
\end{equation}
Next we observe that \eqref{it:a}--\eqref{it:c} are valid also for
\[
Z_Q=\int_{Q\cap L_{x,r}\setminus\Delta(2^{-k})}\mu_{k+1}-\mu_k\,d\mathcal{H}^1\,,
\]
Indeed, for \eqref{it:a} it is enough to observe that \eqref{eq:mart} implies $\EE(\mu_{k+1}(y)\,|\,A_k)=\mu_k(y)$ for $y\in [1,2]^d\setminus\Delta(2^{-k})$ and that $\mu_{k+1},\mu_k\lesssim 2^{dk(1-s)}$. Moreover, the bound on the dependency degree for $Z_Q$ holds for the same reason as for the $X_Q$ and \eqref{it:c} does not involve $X_Q$ nor $Z_Q$, but it is due to \eqref{eq:Y_k_bound}. Whence, Lemma \ref{lem:HoeffdingJanson} implies that we may replace $X_Q$ by $Z_Q$ in \eqref{eq:sumXQ}.

Summing over all $j>k$ in \eqref{eq:sumXQ} yields
\begin{equation}\label{eq:fixed_x_X}
\begin{split}
\PP&\left(\int_{L_{x,r}\setminus\Delta((2^{-k}))}\lambda_k\, d\mathcal{H}^1>C_1 2^{m(d-1)(1-s)/2}2^{-\theta k}\,|\,\mathcal{G}_{k-1}\right)\\
&\le \sum_{j>k}\PP\left(\sum_{\mathbf{I}_j}X_Q>\tfrac{1}{2(j-k)^2}C_1 2^{m(d-1)(1-s)/2-\theta k }\,|\,\mathcal{G}_{k-1}\right)\\
&\lesssim \exp(-c 2^{k(1-d+ds-2\theta)})\,.
\end{split}
\end{equation}
and similarly
\begin{equation}\label{eq:fixed_x_Z}
\begin{split}
&\PP\left(\int_{L_{x,r}\setminus\Delta((2^{-k}))}\mu_{k+1}-\mu_k\, d\mathcal{H}^1 >C_1 2^{m(d-1)(1-s)/2}2^{-\theta k}\,|\,\mathcal{G}_{k-1}\right)\\
&\lesssim \exp(-c 2^{k(1-d+ds-2\theta)})\,.
\end{split}
\end{equation}
More precisely, the above holds uniformly for any realization of $A_{k-1}$ such that $\mathcal{G}_{k-1}$ holds, and in particular simply conditioning on $\mathcal{G}_{k-1}$.
Combining \eqref{eq:fixed_x_X} and \eqref{eq:fixed_x_Z} with Lemma \ref{lem:diag_transversality}, we thus have for each fixed $(x,r)\in\Gamma_m\times[1,2]^d$,
\begin{equation}\label{eq:fixed_x_total}
\begin{split}
\PP&\left(\max\{X_k(x,r),Z_k(x,r)\}>C_1 2^{ms+k(d-1+ds)}+ C_1 2^{m(d-1)(1-s)/2}2^{-\theta k}\,|\,\mathcal{G}_{k-1}\right)\\
&\lesssim \exp(-c 2^{k(1-d+ds-2\theta)})\,.
\end{split}
\end{equation}

To complete the proof, we still need to show that this estimate holds simultaneously for all $(x,r)\in\Gamma_m\times[1,2]^d$. The reason for this is that while \eqref{eq:fixed_x_total} gives an estimate that is superexponentially small in $k$, we may approximate $\sup_{x,r} X_k(x,r)$ (resp. $\sup_{x,r} Z_k(x,r)$) with a discrete family whose size grows only exponentially in $k$.
To that end, we will first derive a deterministic continuity modulus for the maps $(x,r)\mapsto X_k(x,r)$, $(x,r)\mapsto Z_k(x,r)$. Note that since $r_i\in[1,2]$, all the lines $L_{x,r}$ form an angle $\gtrsim 1$ with the coordinate hyperplanes $\{x\in\R^d\,:\,x_i=0\}$. From this, it easily follows that for all $Q\in\mathcal{Q}_{k}$, the map $(x,r)\mapsto\mathcal{H}^{1}(Q\cap L_{x,r})$ defined on $[-4,0]^d\times[1,2]^d$, is Lipschitz with a Lipschitz constant independent of $Q$ (and $k$).
Furthermore, each $L_{x,r}$ intersects  $\lesssim 2^{ks}$ such cubes. Taking once more into account the bound \eqref{eq:trivial_up},
we deduce that
\begin{align*}
|X_k(x,r)-X_k(x',r')|\lesssim 2^{k(d+s-ds)}(|x-x'|+|r-r'|)\,,\\
|Z_k(x,r)-Z_k(x',r')|\lesssim 2^{k(d+s-ds)}(|x-x'|+|r-r'|)\,.
\end{align*}
Thus, if $\Lambda_k\subset\Gamma_m\times[1,2]^d$ is $2^{-Ck}$-dense, for some sufficiently large constant $C$, then
\[\sup_{x\in\Gamma_m, r\in[1,2]^d}X_k(x,r)\le\sup_{(x,r)\in\Lambda_k}X_k(x,r)+\frac12 C_1 2^{ms+k(d-1+ds)}\,, \]
and similarly for $Z_k(x,r)$.
As $\Lambda_k$ may be chosen to have at most $2^{Ck}$ elements, recalling \eqref{eq:fixed_x_total} gives
\[
1-\PP(\mathcal{G}_{k}\,|\,\mathcal{G}_{k-1})\lesssim 2^{C k}\exp(-c 2^{k(1-d+ds -2\theta)}) \le 2^{C k}\exp(-c 2^{\delta k}) \,,
\]
using that $\delta<1-d+ds-2\theta$. We conclude that, for each $n\ge m-1$,
\[
1-\PP(\mathcal{G}_n)   \le \sum_{k=m}^{n} 1 -\PP(\mathcal{G}_k|\mathcal{G}_{k-1}) \lesssim \exp(-c 2^{\delta m})\,.
\]
This shows that \eqref{eq:claim-G-k} holds and completes the proof.
\end{proof}

\begin{proof}[Proof of Lemma \ref{lem:no_intersections}]
The proof is a small modification of that of Lemma \ref{lem:C_k_small_gen}. Fix $K\ge 1$ and $m\in\N$.
For all $k\ge m$, let $\mathcal{G}_k$ now denote the event that
\[
X_n(x,r), Z_n(x,r)\le K 2^{ms}2^{n\xi}\text{ for all }m\le n\le k, x\in\Gamma_m, r\in[1,2]^d,
\]
and, as before, let $\mathcal{G}_{m-1}$ be the sure event. We shall estimate $\PP(\mathcal{G}_{k}\,|\,\mathcal{G}_{k-1})$ for $k\ge m$.

Fix $(x,r)\in\Gamma_m\times[1,2]^d$. As before, we write
\begin{align*}
X_k(x,r)&=\int_{L_{x,r}\cap\Delta(2^{-k})}\lambda_k\,d\mathcal{H}^1+\int_{L_{x,r}\setminus\Delta(2^{-k})}\lambda_k\,d\mathcal{H}^1\,,\\
Z_k(x,r)&=\int_{L_{x,r}\cap\Delta(2^{-k})}(\mu_{k+1}-\mu_k)\,d\mathcal{H}^1+\int_{L_{x,r}\setminus\Delta(2^{-k})}(\mu_{k+1}-\mu_k)\,d\mathcal{H}^1,
\end{align*}
and observe that, by Lemma \ref{lem:diag_transversality}, it is enough to bound the second term in both decompositions.

Fix $(x,r)\in\Gamma_m\times [1,2]$ and a realization of $A_{k-1}$ such that $\mathcal{G}_{k-1}$ holds. Recall the definition of the random variables $Y_n$ from \eqref{eq:def-Y-n}. Using Lemma \ref{lem:Y_k_triv_bound} and telescoping as in the proof of Lemma \ref{lem:C_k_small_gen},
\[
Y_k(x,r)=Y_{m}(x,r)+\sum_{n=m}^{k-1}Z_n(x,r) \lesssim_{\xi} 2^{ms}2^{k\xi}\,.
\]
We now define the random variables $X_Q$ and $Z_Q$ and the families $\mathbf{I}, \mathbf{I}_j$ as in the proof of Lemma \ref{lem:C_k_small_gen}. Then \eqref{it:a}--\eqref{it:b} continue to hold for the same reasons, while \eqref{it:c} now becomes
\begin{enumerate}
  \item[(c')] \label{it:c'} $\#\mathbf{I}_j \lesssim_{\xi} 2^{ms+k(\xi-(1-s)d)+j}$.
\end{enumerate}
Indeed, we have
\[
\#\mathbf{I}_j \cdot 2^{-j} \cdot 2^{k(1-s)d} \lesssim Y_k(x,r) \lesssim_{\xi}   2^{ms}2^{k\xi}.
\]
Applying Lemma \ref{lem:HoeffdingJanson} as in the proof of Lemma \ref{lem:C_k_small_gen}, we get
\begin{align*}
\PP\left(\int_{L_{x,r}\setminus\Delta(2^{-k})}\lambda_k>2^{ms/2} 2^{\xi k} \,|\,\mathcal{G}_{k-1}\right)\lesssim_{\xi} \exp(-c 2^{k(\xi+1-d+ds)})\,,\\
\PP\left(\int_{L_{x,r}\setminus\Delta(2^{-k})}(\mu_{k+1}-\mu_k)>2^{ms/2} 2^{\xi k} \,|\,\mathcal{G}_{k-1}\right)\lesssim_{\xi} \exp(-c 2^{k(\xi+1-d+ds)})\,,
\end{align*}
uniformly (over all the realizations of $\mathcal{G}_{k-1}$). While this estimate holds for a fixed $(x,r)$, we can extend it to all $(x,r)\in\Gamma_m\times[1,2]^d$ simultaneously using the same argument from Lemma \ref{lem:C_k_small_gen} (in fact because we only seek an exponential upper bound, the argument is easier in the current setting). The proof is now concluded in exactly the same way as the proof of Lemma \ref{lem:C_k_small_gen}: for $n \ge m-1$,
\[
1-\PP(\mathcal{G}_n)   \le \sum_{k=m}^{n} 1 -\PP(\mathcal{G}_k|\mathcal{G}_{k-1}) \lesssim_\xi \exp(-c 2^{\delta m})\,.
\]
\end{proof}

\begin{rem}
Since for ``most'' lines $L$ the intersections $L\cap A_n^d$ are empty, one could hope to improve the exponent $\xi$ in Lemma \ref{lem:no_intersections} to some value $<d-1-ds$. However, at least when $|\Omega|\approx 1$, this is not possible: Given an interval $D\subset A_n$ (an element of $\mathcal{D}_n$), we can select $r(t)$ for all $t\in I\subset[0,1]$ with $|I|\gtrsim 1$ in such a way that $L_{x,r(x)}$ passes through the center of $D^d$. Indeed, for a suitable $z=z_D$, define  $r(t)$ such that $(z-t)/r(t)$ equals the center point of $D$. Then,
\begin{align*}
\int_{[0,1]^d}\int\prod_{i=1}^d\sigma_n(\phi(x_i,r(x_i),z))\,dz\,dx\gtrsim  2^{n(d-1-ds)}\,.
\end{align*}
\end{rem}


\end{document}